\documentclass{amsart}
\usepackage{amsmath}
\usepackage{amssymb}
\usepackage{amsthm}
\usepackage{graphicx}
\usepackage{url}
\usepackage{colordvi}
\usepackage{color}
\usepackage{verbatim}

\newcommand{\fm}{{\mathfrak m}}
\newcommand{\R}{{\mathbb R}}

\newcommand{\C}{\mathbb{C}}

\newcommand{\rx}{\mathbb{R}[x_1,\ldots,x_n]}

\newtheorem{theorem}{Theorem}[section]
\newtheorem{lemma}[theorem]{Lemma}

\newtheorem{proposition}[theorem]{Proposition}
\newtheorem{corollary}[theorem]{Corollary}
\theoremstyle{definition}
\newtheorem{definition}[theorem]{Definition}
\theoremstyle{remark}
\newtheorem{remark}[theorem]{Remark}
\newtheorem{example}[theorem]{Example}

\DeclareMathOperator{\Hom}{Hom}
\DeclareMathOperator{\conv}{conv}

\begin{document}
\title{Strong Nonnegativity and Sums of Squares On Real Varieties}

\author{Mohamed Omar}
\author{Brian Osserman}

\begin{abstract}Motivated by scheme theory, we introduce strong nonnegativity 
on real varieties, which has the property that a sum of squares is strongly
nonnegative.
We show that this algebraic property is equivalent to nonnegativity for
nonsingular real varieties.  Moreover, for singular varieties, we reprove 
and generalize obstructions of Gouveia and Netzer to 
the convergence of the theta body hierarchy of convex bodies approximating 
the convex hull of a real variety.
\end{abstract}

\thanks{The first author was partially supported by NSERC PGS-D 281174 and NSF grant DMS-0914107.}
\maketitle

\section{Introduction}
The relationship between nonnegative polynomials and sums of squares of
polynomials on real varieties is a classical subject, dating back to Hilbert.
In real algebraic geometry, a large body of research is dedicated to
understanding the gap between these families.  At the same time, this
subject has recently become important in the emerging field of convex
algebraic geometry, where it is relevant to the effectiveness of computing
convex hulls of algebraic varieties.  This in turn has been intimately related
to the geometry of feasible regions of semidefinite programs (see 
\cite{GouveiaThomas2012} and references therein).  Motivated by this and inspired
by scheme theory, we introduce an intermediate class of polynomials
which we call \emph{strongly nonnegative}.  This class is particularly useful
for understanding the role that singularities on real varieties play in
obstructing sums of squares representations.

We begin by exploring the basic properties of strong nonnegativity, showing
in particular in Theorem \ref{thm:nonneg-nbhd} 
that strong nonnegativity at a point implies nonnegativity in a neighborhood 
of that point, and that the converse holds for nonsingular points.
In the singular case, we
study obstructions to the theta body hierarchy \cite{GouveiaParriloThomas2008}
of convex bodies approximating the convex hull of a real variety.  The strength
of this approximation is governed by the sums of squares representability of
linear functions on a variety.  We are able to recover very transparently in
Theorem \ref{thm:not-theta-exact} the obstructions produced by Gouveia and
Netzer in \cite{GouveiaNetzer} to convergence of this hierarchy. The same
argument gives us Corollary \ref{cor:thetaobstruct}, a generalized version of
their obstruction. Finally, Proposition \ref{prop:strictsos} shows that our
construction behaves well in the context of the foundational constructions of
Gouveia, Parrilo and Thomas in \cite{GouveiaParriloThomas2008}.

\subsection*{Acknowledgements}

We would like to thank Rekha Thomas and Jo\~ao Gouveia for helpful 
conversations. In particular, Examples \ref{ex:squiggly} and \ref{ex:sphere} 
were found in consultation with them. We would also like to thank the referee 
for helpful comments.

\section{Strong nonnegativity}

Our convention throughout, given an ideal $I \subseteq \rx$, is to use
$V_{\R}(I)$ for the real vanishing set of $I$, and use $V(I)$ in relation to 
concepts depending on the ring $\rx/I$, which we will denote by $A$.  
Formally, $V(I)$ is the closed subscheme $\mbox{Spec}(A)\subseteq \R^n$, but 
our definitions will be in terms of $A$, so no knowledge of schemes is
required. All of our ring homomorphisms are assumed to be $\R$-algebra
homomorphisms.

We begin by introducing our stricter definition of nonnegativity. Our 
motivating example is the following:

\begin{example}\label{ex:xsquared}
\begin{figure}
\def\JPicScale{0.8}
\ifx\JPicScale\undefined\def\JPicScale{1}\fi
\unitlength \JPicScale mm
\begin{picture}(160,16)(0,0)
\linethickness{0.1mm}
\put(0,6){\line(1,0){160}}
\put(80,2){\makebox(0,0)[cc]{$0$}}

\put(80,6){\makebox(0,0)[cc]{}}

\linethickness{0.3mm}
\put(60,6){\line(1,0){40}}
\put(100,6){\vector(1,0){0.12}}
\put(60,6){\vector(-1,0){0.12}}
\put(80,16){\makebox(0,0)[cc]{$V(x^2)$}}

\linethickness{0.3mm}
\multiput(60,7)(1.9,0){11}{\line(1,0){0.95}}
\put(60,7){\vector(-1,0){0.12}}
\put(61,10){\makebox(0,0)[cc]{$-\epsilon$}}



\end{picture}
\caption{Motivating example}
\end{figure}
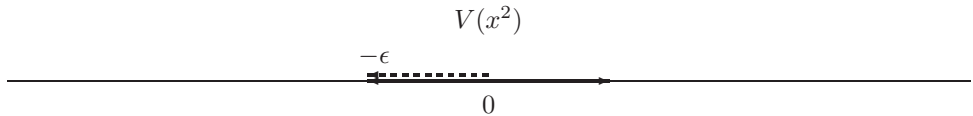

Suppose $I \subseteq \R[x]$ is the ideal
generated by $x^2$. Then set-theoretically, we have $V_{\R}(I)$ equal to the
origin. Thus, the function $x$ is nonnegative on $V_{\R}(I)$. However, one 
easily checks that $x$ is not a sum of squares modulo $I$. 

From a more scheme-theoretic perspective, we should think of $V(I)$ as
not consisting only of the origin, but also including an infinitesimal
thickening in both directions -- in particular, in the negative direction.
Thus, we should not think of $x$ as being nonnegative on the \textit{scheme} 
$V(I)$.
\end{example}

Recall that if $I \subseteq \rx$ is an ideal, then the points of
$V_{\R}(I)$ correspond precisely to ($\R$-algebra) homomorphisms $A \to \R$,
where $A=\rx/I$. The homomorphism obtained from a given $P \in V_{\R}(I)$
is simply given by evaluating polynomials at $P$.
Thus, one may rephrase nonnegativity as saying that $f$ is nonnegative
if its image under any homomorphism $A \to \R$ is nonnegative. Our
definition will consider a broader collection of such homomorphisms.
In particular, given a point of $V_{\R}(I)$ corresponding to 
$\varphi:A \to \R$, it is standard that the (scheme-theoretic) tangent 
space to $V(I)$ at the point is in bijection with homomorphisms
$A \to \R[\epsilon]/(\epsilon^2)$ which recover $\varphi$ after composing
with the unique homomorphism $\R[\epsilon]/(\epsilon^2) \to \R$, which
necessarily sends $\epsilon$ to $0$.

In Example \ref{ex:xsquared}, a tangent vector in the ``negative direction''
is given by the homomorphism 
$\R[x]/(x^2) \to \R[\epsilon]/(\epsilon^2)$ sending $x$ to $-\epsilon$. If
we consider $-\epsilon$ to be ``negative'', we may thus consider the 
function $x$ to take a negative value on this tangent vector to $V(I)$.
We formalize and generalize this idea by considering also higher-order
infinitesimal arcs, as follows.

\begin{definition}
Given $f \in \R[\epsilon]/(\epsilon^m)$, $f = a_0 + a_1 \epsilon + \cdots + a_{m-1} \epsilon^{m-1}$,  we say $f$ is \textbf{nonnegative} if $f=0$, or $a_N > 0$ where $N = \min \{j \ : \ a_j \neq 0\}$.
\end{definition}

Note that $\R[\epsilon]/(\epsilon^m)$ has a unique homomorphism to $\R$,
necessarily sending $\epsilon$ to $0$. We say that 
$\varphi:A \to \R[\epsilon]/(\epsilon^m)$ is \textbf{at $P$} for (a
necessarily unique) $P \in V_{\R}(I)$ if $P$ is the point corresponding
to the composed homomorphism $A \to \R$. 

\begin{definition}
Let $I \subseteq \rx$ be an ideal, and $A:=\rx/I$. Given $P \in V_{\R}(I)$,  we say $f \in A$ is \textbf{strongly nonnegative} at $P$ if for every $m \geq 0$ and for every $\R$-algebra homomorphism 
\[
\varphi: A \to \R[\epsilon]/(\epsilon^m)
\]
at $P$, we
have $\varphi(f)$ is nonnegative. We say $f$ is \textbf{strongly nonnegative}
on $V(I)$
if it is strongly nonnegative at $P$ for all $P \in V_{\R}(I)$.
\end{definition}

We begin with some basic observations on the property of strong nonnegativity.

\begin{proposition}\label{prop:basic} Given $f \in A$, we have the following
statements.
\begin{enumerate}
\item If $f$ is strongly nonnegative
at $P \in V_{\R}(I)$, then $f$ is nonnegative at $P$.
\item If $f$ is strictly positive at $P \in V_{\R}(I)$, then $f$ 
is strongly nonnegative at $P$.
\item If $f$ is a sum of squares, then $f$ is strongly nonnegative.
\end{enumerate}
\end{proposition}

\begin{proof} We obtain (1) immediately by setting $m=1$ in the definition,
since this yields the evaluation map at $P$.

For (2), given any homomorphism $\varphi:A \to \R[\epsilon]/(\epsilon^m)$
at $P$, by definition we have that composing with 
$\R[\epsilon]/(\epsilon^m) \to \R$ gives the evaluation map at $P$, 
under which
$f$ is strictly positive by hypothesis. But then if we write
$\varphi(f) = a_0 + a_1 \epsilon + \cdots + a_{n-1} \epsilon^{n-1}$, we 
must have $a_0=f(P)>0$, and thus $\varphi(f)$ is nonnegative. Since $\varphi$
was arbitrary at $P$, we conclude $f$ is strongly nonnegative at $P$.

Finally, for (3) if $f = \sum_{i=1}^r h_i^2$, and 
$\varphi: A \to \R[\epsilon]/(\epsilon^m)$ is an $\R$-algebra homomorphism, 
then the leading term of each ${(\varphi(h_i))}^2$ is nonnegative, and hence 
so is that of $\varphi(f)$.
\end{proof}

We will show in Theorem \ref{thm:nonneg-nbhd} that
in fact if $f$ is strongly nonnegative at $P$, then it is nonnegative on
a neighborhood of $P$, and that the converse holds if $P$ is a nonsingular
point of $V(I)$. Of course, the converse does not hold in general.

\begin{example} 
\begin{figure}
\def\JPicScale{0.8}
\ifx\JPicScale\undefined\def\JPicScale{1}\fi
\unitlength \JPicScale mm
\begin{picture}(160,80)(0,0)
\linethickness{0.1mm}
\put(80,0){\line(0,1){80}}
\linethickness{0.1mm}
\put(40,40){\line(1,0){80}}
\linethickness{0.3mm}
\qbezier(60,60)(65.19,49.56)(70,44.75)
\qbezier(70,44.75)(74.81,39.94)(80,40)
\put(60,60){\vector(-1,2){0.12}}
\qbezier(80,40)(85.19,39.94)(90,44.75)
\qbezier(90,44.75)(94.81,49.56)(100,60)
\put(100,60){\vector(1,2){0.12}}
\put(80,40){\makebox(0,0)[cc]{}}

\linethickness{0.3mm}
\multiput(80,40.5)(2,0){1}{\line(1,0){1}}
\linethickness{0.3mm}
\multiput(81.5,40.5)(1.6,0.2){1}{\multiput(0,0)(0.8,0.1){1}{\line(1,0){0.8}}}
\linethickness{0.3mm}
\multiput(82.8,40.7)(1.4,0.4){1}{\multiput(0,0)(0.35,0.1){2}{\line(1,0){0.35}}}
\linethickness{0.3mm}
\multiput(83.9,41.1)(1.4,0.6){1}{\multiput(0,0)(0.23,0.1){3}{\line(1,0){0.23}}}
\linethickness{0.3mm}
\multiput(85,41.6)(1.2,0.6){1}{\multiput(0,0)(0.2,0.1){3}{\line(1,0){0.2}}}
\linethickness{0.3mm}
\multiput(86,42.1)(1.2,0.8){1}{\multiput(0,0)(0.2,0.13){3}{\line(1,0){0.2}}}
\linethickness{0.3mm}
\multiput(87,42.8)(1.2,1){1}{\multiput(0,0)(0.15,0.12){4}{\line(1,0){0.15}}}
\linethickness{0.3mm}
\multiput(88,43.6)(1,1){1}{\multiput(0,0)(0.12,0.12){4}{\line(1,0){0.12}}}
\linethickness{0.3mm}
\multiput(88.8,44.4)(1,1){1}{\multiput(0,0)(0.12,0.12){4}{\line(1,0){0.12}}}
\linethickness{0.3mm}
\multiput(89.6,45.2)(1,1){1}{\multiput(0,0)(0.12,0.12){4}{\line(1,0){0.12}}}
\linethickness{0.3mm}
\multiput(90.4,46)(1,1.2){1}{\multiput(0,0)(0.12,0.15){4}{\line(0,1){0.15}}}
\put(90.9,46.6){\vector(3,4){0.12}}
\put(53,74){\makebox(0,0)[cc]{$V(y-x^2,y^2)$}}

\put(71,32){\makebox(0,0)[cc]{$(0,0)$}}

\put(103,51){\makebox(0,0)[cc]{$x \mapsto \epsilon$}}

\put(104,47){\makebox(0,0)[cc]{$y \mapsto \epsilon^2$}}

\end{picture}
\caption{$-y$ is not strongly nonnegative on $V(y-x^2,y^2)$}
\end{figure}
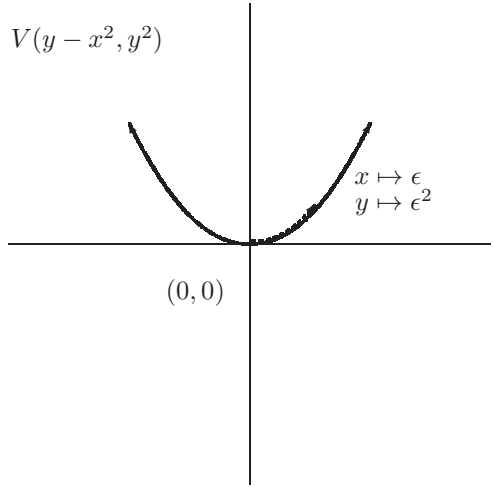

Consider $I=(y-x^2,y^2) \subseteq \R[x,y]$, and $P=(0,0)$ the only point
of $V_{\R}(I)$. Then $-y$ is not strongly nonnegative on $V(I)$: 
under the homomorphism $\varphi:\R[x,y]/I \to \R[\epsilon]/(\epsilon^3)$
at $P$ sending $x$ to $\epsilon$ and $y$ to $\epsilon^2$, we have
$\varphi(-y)=-\epsilon^2$ is not nonnegative. 

On the other hand, $y$ is strongly nonnegative on $V(I)$ by Proposition 
\ref{prop:basic} (3), since $y=x^2$ modulo $I$.
\end{example}

We also give an example where $V(I)$ is reduced (i.e., $I$ is radical)
for which strong nonnegativity is strictly stronger than nonnegativity.

\begin{example}\label{ex:cusp}
\begin{figure}
\def\JPicScale{0.8}
\ifx\JPicScale\undefined\def\JPicScale{1}\fi
\unitlength \JPicScale mm
\begin{picture}(160,80)(0,0)
\linethickness{0.1mm}
\put(80,0){\line(0,1){80}}
\linethickness{0.1mm}
\put(40,40){\line(1,0){80}}
\linethickness{0.3mm}
\qbezier(80,40)(90.38,39.88)(100,49.5)
\qbezier(100,49.5)(109.62,59.12)(120,80)
\linethickness{0.3mm}
\qbezier(80,40)(90.38,39.6)(100,29.97)
\qbezier(100,29.97)(109.62,20.35)(120,0)
\linethickness{0.3mm}
\multiput(65,40)(2,0){8}{\line(1,0){1}}
\put(65,40){\vector(-1,0){0.12}}
\put(64,34){\makebox(0,0)[cc]{$x \mapsto -\epsilon$}}

\put(62.5,28){\makebox(0,0)[cc]{$y \mapsto 0$}}

\put(60,65){\makebox(0,0)[cc]{$V(y^2-x^3)$}}

\end{picture}
\caption{The negative direction $(-1,0)$ at $(0,0)$ on $V(y^2-x^3)$}
\end{figure}
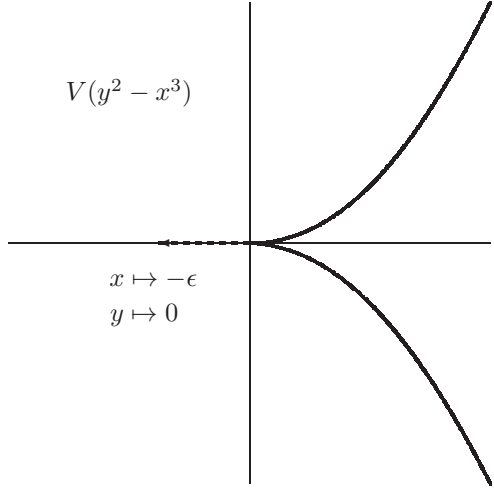

Consider the ideal $I = (y^2 - x^3) \subseteq \R[x,y]$, and the function $f(x,y)=x$ nonnegative on $V_{\R}(I)$.  Note that $f(x,y)$ is negative on the direction $(-1,0)$ at the singular point $(0,0)$ of $V(I)$.  This is realized algebraically by the homomorphism 
\[
\varphi: \R[x,y]/(y^2-x^3) \to \R[\epsilon]/(\epsilon^2), \ \ \varphi(x) = -\epsilon \ \ \varphi(y)=0,
\]
at $P$,
which proves $f$ is not strongly nonnegative since the leading coefficient of $\varphi(f)=\varphi(x)=-\epsilon$ is negative.  Thus, $f$ cannot be a sum of squares.

This example may be made compact by instead setting $I = (y^2-x^3+x^4)$.
\end{example}

\begin{remark} A suitable local version of Proposition \ref{prop:basic} (3)
may be described in terms of the complete local ring $\hat{A}_P$ of $V(I)$
at $P$. Specifically, if $f$ is a sum of squares in $\hat{A}_P$, then $f$
is strongly nonnegative at $P$.
The proof is the same, since any homomorphism 
$A \to \R[\epsilon]/(\epsilon^m)$ at $P$ factors through the complete
local ring.
\end{remark}

\begin{remark}
Note that if there exists a homomorphism $A \to \R[\epsilon]/(\epsilon^m)$
such that the image of $f$ has its leading term in odd degree,
then $f$ is not strongly nonnegative, since we may change the sign of the
coefficient by composing with the automorphism of $\R[\epsilon]/(\epsilon^m)$
sending $\epsilon$ to $-\epsilon$.
\end{remark}

We now consider the deeper question of the relationship between strong
nonnegativity at a point, and nonnegativity in a neighborhood. This 
requires concepts related to nonsingularity, which for the sake of clarity,
we now recall:

\begin{definition}
Given $I \subseteq\rx$ and $P \in V_{\R}(I)$, set $A=\rx/I$, and let 
$\fm_P \subseteq A$ be the maximal ideal of $A$ consisting of polynomials 
vanishing at $P$.
Note that because $P \in V_{\R}(I)$, we have $A/\fm_P \cong \R$.
The \textbf{cotangent
space} of $V(I)$ at $P$ is the real vector space $\fm_P/\fm_P^2$, and
the \textbf{tangent space} of $V(I)$ at $P$ is the dual space 
$\Hom_{\R}(\fm_P/\fm_P^2, \R)$. The \textbf{dimension} of $V(I)$ at $P$
is the dimension of the local ring $A_{\fm_P}$. Finally,
$V(I)$ is \textbf{nonsingular} at $P$ if the tangent space at $P$ has
dimension equal to the dimension of $V(I)$ at $P$.
\end{definition}

We then have the following:

\begin{theorem}\label{thm:nonneg-nbhd}
Given $I \subseteq \rx$ and a point $P \in V_{\R}(I)$, suppose that 
$f\in A:=\rx/I$ is strongly nonnegative at $P$. Then $f$ is
nonnegative in a (real) neighborhood of $P$. Moreover, the converse holds
if $P$ is a nonsingular point of $V(I)$.
\end{theorem}

It will be convenient to extend our terminology as follows:

\begin{definition} Suppose $P \in V_{\R}(I)$. Then a homomorphism
$\varphi:A \to \R[[t]]$ is \textbf{at $P$} if the preimage of the ideal
generated by $t$ is the (maximal) ideal of functions vanishing at $P$.
\end{definition}

The following is the main technical lemma, which does not involve strong 
nonnegativity and which applies without the nonsingularity hypothesis. 
Although the result is well known in real algebraic geometry, we include
it for the convenience of the reader.

\begin{lemma}\label{lem:arcs}
Given $I \subseteq \rx$ a point $P \in V_{\R}(I)$, and $f \in A:=\rx/I$, 
the following are equivalent:
\begin{enumerate}
\item $f$ is nonnegative in a (real) neighborhood of $P$;
\item for every homomorphism $\varphi:A \to \R[[t]]$ at $P$ taking values
in locally convergent power series, we have that the leading term of 
$\varphi(f)$ is nonnegative.
\item for every homomorphism $\varphi:A \to \R[[t]]$ at $P$, we have that
the leading term of $\varphi(f)$ is nonnegative.
\end{enumerate}
\end{lemma}

Geometrically, a homomorphism $A \to \R[[t]]$ at $P$ taking values in 
locally convergent power series defines an
\emph{analytic arc} at $P$; that is, we obtain an analytic map from 
$(-c,c) \subseteq \R$ to $V(I)$ for some $c>0$, sending $0$ to $P$. By
analogy, we think of an arbitrary homomorphism $A \to \R[[t]]$ at $P$
as a \emph{formal arc} at $P$. Thus, the content of the lemma may be
viewed as saying that $f$ is nonnegative on a neighborhood of $P$ if and
only if it is nonnegative on every analytic arc at $P$, if and only if
it is nonnegative on every formal arc at $P$.

\proof We first show that (1) and (2) are equivalent. The implication that
(1) implies (2) is straightforward. Indeed, if $\varphi(f)$ has negative
leading term for some $\varphi$, then for $t_0$ sufficiently small and 
positive, we would have $\varphi(f)(t_0)<0$, and because 
$$\varphi(f)(t_0)=f(\varphi(x_1)(t_0),\dots,\varphi(x_n)(t_0)),$$
the points $(\varphi(x_1)(t_0),\dots,\varphi(x_n)(t_0))$ would yield points
arbitrarily close to $P$ with $f$ negative.  For the converse, we appeal to 
the Curve Selection Lemma (see Theorem VII.4.2 and Remarks VII.4.3 of 
\cite{ConstructibleSets}). Suppose that (1) is false.  Then $P$ is
in the closure of the set $S = \{x \in V_{\R}(I) \ : \ f(x) < 0\}$.  
Now, $S$ is semi-algebraic, so by the Curve Selection Lemma, there exists 
a half-branch at $P$ of an algebraic curve contained in $V(I)$ such that
away from $P$, the half-branch is contained in $S$. This half-branch is
in particular analytic, so it is defined by a
homomorphism $\varphi: A \to \R[[t]]$ at $P$ taking values in locally
convergent power series, and moreover we have that
$\varphi(f)$ is negative for all sufficiently small positive values of $t$.
We conclude that $\varphi(f)$ has negative leading coefficient, as desired.

We now move on to proving the equivalence of (2) and (3).  Of course, (3)
trivially implies (2). The key ingredient for the converse is an Artin-style 
approximation theorem. Suppose we have $\varphi:A \to \R[[t]]$ at $P$ such
that $\varphi(f)$ has negative leading term. A theorem of Greenberg 
\cite{Greenberg} (which is a special case of Artin's approximation theorem;
see also \S VII.3 of \cite{ConstructibleSets})
asserts that we can replace $\varphi$ by a homomorphism $\varphi'$
which takes values in locally convergent power series and agrees with
$\varphi$ to arbitrarily high order; that is, for any fixed $N$, we
can find $\varphi'$ such that for all $g \in A$, we have that the
first $N$ terms of $\varphi'(g)$ agree with the first $N$ terms of 
$\varphi(g)$. In particular, we may choose 
$\varphi'$ such that $\varphi'(f)$ still has negative leading term, and
we thus conclude the desired result.
\qed



The proof of Theorem \ref{thm:nonneg-nbhd} is almost immediate from
Lemma \ref{lem:arcs}.

\begin{proof}[Proof of Theorem \ref{thm:nonneg-nbhd}]
First suppose that $f$ is not nonnegative on any neighborhood of $P$. Then 
Lemma \ref{lem:arcs} implies that there exists a
homomorphism $A \to \R[[t]]$ under which $f$ has negative leading term.
If the leading term occurs in degree $m-1$, truncating from $\R[[t]]$
to $\R[\epsilon]/(\epsilon^{m})$ via $t \mapsto \epsilon$ then shows that 
$f$ is not strongly nonnegative.

Conversely, suppose that
$f$ is nonnegative on a neighborhood of $P$ in $V_{\R}(I)$, and $V(I)$ is
nonsingular at $P$. Because nonsingularity is equivalent to smoothness
in characteristic $0$, by a generalization of Hensel's lemma 
if we have a homomorphism 
$\varphi:A \to \R[\epsilon]/(\epsilon^m)$ at $P$, we can lift to 
$\R[\epsilon]/(\epsilon^{m'})$ for $m'$ arbitrarily large
(see Proposition 2.2.15 and Proposition 2.2.6 of 
\cite{BoschLutkebohmertRaynaud}). 
Passing to the limit as $m'$ goes to $\infty$, we obtain a homomorphism
$\widetilde{\varphi}:A \to \R[[t]]$ lifting $\varphi$. It follows
from Lemma \ref{lem:arcs} that $\widetilde{\varphi}(f)$ must
either be $0$ or have positive leading coefficient, and we thus conclude
the same for $\varphi(f)$. Thus, $f$ is strongly nonnegative.
\end{proof}

\section{Obstructions to sums of squares}

We now apply the concept of strong nonnegativity to study obstructions
to nonnegative functions being sums of squares. We will use the concept
of degrees of functions, and consequently from this point on the choice
of imbedding of $V(I)$ into affine space becomes relevant.
Recall the following definition:

\begin{definition} Fix $I \subseteq \rx$. For $f \in \rx$ and $k \geq 1$, 
we say that $f$ is $k$-sos modulo $I$ if 
there exist $g_1,\dots,g_m \in \rx$ of degree at most $k$ such that 
$$f \equiv \sum_{i=1}^m g_i^2 \pmod{I}.$$
Given $d,k \geq 1$, 
we say that $I$ is \textbf{$(d,k)$-sos} if every $f \in \rx$ 
of degree at most $d$ which is nonnegative on $V_{\R}(I)$ is $k$-sos modulo
$I$.
\end{definition}

Note that if $f$ is $k$-sos modulo $I$, then $f$ is nonnegative on 
$V_{\R}(I)$, so the latter definition says that as many functions as 
possible (of degree at most $d$) are $k$-sos modulo $I$.

Proposition \ref{prop:basic} (3) then trivially implies:

\begin{corollary}\label{cor:sos-obstruct}
Let $I \subseteq \rx$ be an ideal.  If there exists a function $f\in A$
of degree less than or equal to $d$ which is nonnegative on $V_{\R}(I)$ but 
not strongly nonnegative, then $I$ is not $(d,k)$-sos for any $k$.
\end{corollary}

We now specialize to linear functions, and
recover an obstruction theorem of Gouveia and Netzer; see
Theorem 4.5 of \cite{GouveiaNetzer}. To give the statement, we define:

\begin{definition} A point $P \in V_{\R}(I)$ is \textbf{convex-singular} if it
is a singular point of $V(I)$, it lies on the relative boundary of 
$\conv(V_{\R}(I))$, and the tangent space to $V(I)$ at $P$ meets the 
relative interior of $\conv(V_{\R}(I))$. 
\end{definition}

\begin{remark}
Note that the tangent space of $V(I)$ at $P$ is canonically a subspace of 
the tangent space at $P$ of the ambient affine space $\R^n$, which is 
canonically identified via translation with $\R^n$ itself. Thus the 
definition makes sense.
\end{remark}

\begin{remark} 
Our definition differs slightly from that of \cite{GouveiaNetzer}, which
considers instead the tangent space of $V(\sqrt[\R]{I})$, where $\sqrt[\R]{I}$
is the real radical ideal associated to $I$. For instance, the origin
in $\R^3$ is convex-singular in $V(x^2+y^2)$ in our definition, but not in
\cite{GouveiaNetzer}. Indeed, we consider a point to be its own relative
interior, so for us the origin in $\R^2$ is also convex-singular in 
$V(x^2+y^2)$.
\end{remark}

The obstruction theorem is then the following:

\begin{theorem}\label{thm:not-sos}
Suppose we have $I \subseteq \rx$, and
$P \in V_{\R}(I)$ is convex-singular. Then $I$ is not $(1,k)$-sos for any $k$.
\end{theorem}

\proof 
We claim that there is a linear function $f$ which is nonnegative on
$V_{\R}(I)$, vanishes at $P$, and
induces a nonzero linear function on the tangent space of $V(I)$ at $P$.
In the case that $V_{\R}(I)=\{P\}$, this is trivial: we may take any $f$
whose zero set contains $P$ but not the tangent space at $P$. Thus suppose
$V_{\R}(I)$ is not a single point. 
If we choose a sequence of points in the affine hull of $V_{\R}(I)$ but 
outside $\overline{\conv(V_{\R}(I))}$
converging to $P$, the Separation Theorem (Theorem III.1.3 in 
\cite{barvibook}) gives us a sequence of linear functions on the affine hull,
nonnegative on $V_{\R}(I)$ and negative on the points in our sequence. 
Taking a suitable limit of these (rescaling as necessary) gives a nonzero 
linear function $\bar{f}$ on the affine hull, nonnegative on $V_{\R}(I)$, 
and with $\bar{f}(P)=0$.
We then have that $\bar{f}$ must be strictly positive on the relative 
interior of $V_{\R}(I)$. Choose $f$ to be any lift of $\bar{f}$ to a linear
function on $\R^n$. Now, since $f$ is linear it induces the same function
on the tangent space to $\R^n$ at $P$, and via restriction on the tangent
space to $V(I)$ at $P$. By hypothesis the latter tangent 
space meets the relative interior of $V_{\R}(I)$, so we see that the induced 
function on the tangent space is nonzero, completing the proof of the claim.

Now, because $f$ induces a nonzero linear function on the tangent space, 
there is a tangent vector on which $f$ is negative, and this
corresponds to a homomorphism 
$\varphi:\rx/I \to \R[\epsilon]/(\epsilon^2)$
at $P$ sending $f$ to a negative multiple of $\epsilon$. Thus, $f$ is not
strongly nonnegative. By Corollary \ref{cor:thetaobstruct}, we have that
$f$ is not a sum of squares, and hence $I$ is not $(1,k)$-sos for any $k$.
\qed

Hypersurfaces present a particularly nice case of the theorem.

\begin{corollary}\label{cor:hypersurface} Suppose $I=(g)$ is principal
in $\rx$, and suppose $P \in V_{\R}(I)$ is a singularity lying on the boundary
of $\conv(V_{\R}(I))$. Then $I$ is not $(1,k)$-sos for any $k$.
\end{corollary}

\proof
The variety $V(I)$ has codimension one, so the tangent space at the singular point $0$ is all of $\R^n$. Thus, $P$ is convex-singular, and we conclude
the desired result from Theorem \ref{thm:not-sos}. 
\qed

The following example is a basic example of applying the theorem on
convex singularities.

\begin{example}\label{ex:cusp-2}
Consider the ideal $I = (y^2 - x^3) \subseteq \R[x,y]$ from Example~
\ref{ex:cusp}.  The singular point $P=(0,0)$ of $V(I)$ lies on the boundary 
of $\conv(V_{\R}(I))$, so by Corollary~\ref{cor:hypersurface} we have that 
$I$ is not $(1,k)$-sos for any $k$.  Of course, this also follows from
Corollary~\ref{cor:sos-obstruct} and Example~\ref{ex:cusp}.  As in the
earlier example, this may be made compact by instead setting 
$I = (y^2-x^3+x^4)$.
\end{example}

However, we also see that Corollary~\ref{cor:sos-obstruct} works more generally than for convex singularities. Indeed, convex singularities may be viewed as causing strong nonnegativity to fail at first order, while the general definition requires examining all orders.

\begin{example}\label{ex:squiggly} 
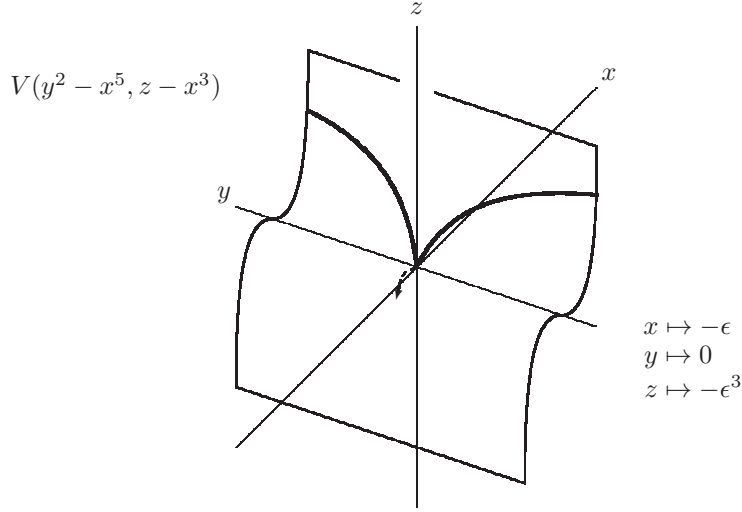
\begin{figure}
\def\JPicScale{0.8}
\ifx\JPicScale\undefined\def\JPicScale{1}\fi
\unitlength \JPicScale mm
\begin{picture}(160,80)(0,0)
\linethickness{0.1mm}
\put(80,0){\line(0,1){80}}
\linethickness{0.1mm}
\multiput(80,40)(0.12,0.12){250}{\line(1,0){0.12}}
\put(80,83){\makebox(0,0)[cc]{$z$}}
\linethickness{0.1mm}
\multiput(50,10)(0.12,0.12){250}{\line(1,0){0.12}}
\put(112,72){\makebox(0,0)[cc]{$x$}}
\linethickness{0.1mm}
\multiput(50,50)(0.36,-0.12){83}{\line(1,0){0.36}}
\linethickness{0.1mm}
\put(48,52){\makebox(0,0)[cc]{$y$}}
\multiput(80,40)(0.36,-0.12){83}{\line(1,0){0.36}}
\put(74,67){\makebox(0,0)[cc]{}}

\put(30,70){\makebox(0,0)[cc]{$V(y^2-x^5,z-x^3)$}}

\put(125,30){\makebox(0,0)[cc]{$x \mapsto -\epsilon$}}
\put(123.5,25){\makebox(0,0)[cc]{$y \mapsto 0$}}
\put(126,20){\makebox(0,0)[cc]{$z \mapsto -\epsilon^3$}}

\linethickness{0.25mm}
\multiput(62,76)(0.36,-0.12){42}{\line(1,0){0.36}}
\linethickness{0.25mm}
\multiput(83,69)(0.36,-0.12){75}{\line(1,0){0.36}}
\linethickness{0.25mm}
\multiput(50,20)(0.36,-0.12){133}{\line(1,0){0.36}}
\linethickness{0.25mm}
\qbezier(62,76)(62.02,61.39)(60.58,54.65)
\qbezier(60.58,54.65)(59.13,47.91)(56,48)
\qbezier(56,48)(52.87,48.09)(51.42,41.35)
\qbezier(51.42,41.35)(49.98,34.61)(50,20)
\linethickness{0.25mm}
\qbezier(110,60)(110.02,45.39)(108.58,38.65)
\qbezier(108.58,38.65)(107.13,31.91)(104,32)
\qbezier(104,32)(100.87,32.09)(99.42,25.35)
\qbezier(99.42,25.35)(97.98,18.61)(98,4)
\linethickness{0.4mm}
\qbezier(80,40)(79.01,49.37)(74.68,55.62)
\qbezier(74.68,55.62)(70.34,61.88)(62,66)
\linethickness{0.4mm}
\qbezier(80,40)(83.58,47.31)(90.8,50.2)
\qbezier(90.8,50.2)(98.02,53.09)(110,52)
\linethickness{0.25mm}
\multiput(78.9,39.6)(0.45,0.15){2}{\line(1,0){0.45}}
\linethickness{0.25mm}
\multiput(77.7,38.4)(0.12,0.18){5}{\line(0,1){0.18}}
\linethickness{0.25mm}
\multiput(77,36)(0.12,0.5){4}{\line(0,1){0.5}}
\put(76.7,35){\vector(-1,-4){0.12}}
\end{picture}
\caption{A higher-order infinitesimal arc on $V(y^2-x^5,z-x^3)$ pointing in the negative direction.}
\end{figure}

Consider the ideal $I = (y^2 - x^5,z-x^3) \subseteq \R[x,y,z]$, and the function $f(x,y,z)=z$ nonnegative on $V_{\R}(I)$.  The only singular point of $V(I)$ is $P=(0,0,0)$, and the tangent space to $V(I)$ at $P$ is precisely the plane $z=0$, so $P$ is not a convex singularity. However, $V(I)$ has higher-order infinitesimal arcs pointing into the negative direction of $z$, for instance given by the homomorphism 
\[
\varphi: \R[x,y,z]/(y^2-x^5,z-x^3) \to \R[\epsilon]/(\epsilon^4), \ \ \varphi(x) = -\epsilon \ \ \varphi(y)=0, \ \ \varphi(z)=-\epsilon^3
\]
at $P$.
Once again, we see that $f$ is not strongly nonnegative, and we conclude by Corollary~\ref{cor:sos-obstruct} that $I$ is not $TH_k$-exact for any $k$.

This example may also be made compact, by setting $I = (y^2-x^5+x^6,z-x^3)$.
\end{example}

However, we see that strong nonnegativity still has limitations in its
ability to recognize functions which are not sums of squares. For instance,
in Example \ref{ex:cusp}, if we took $f=x+c$ for any $c>0$ we would have
a function which is strictly positive, and hence strongly nonnegative,
but still not a sum of squares modulo $I$.
However, Schm\"udgen's Positivstellensatz implies 
(see Corollary 3 of \cite{Schmudgen}) 
that if $V_{\R}(I)$ is compact and $f$ is strictly
positive, then $f$ is a sum of squares. Since strong nonnegativity lies
between nonnegativity and strict positivity, it is natural to wonder if
a strongly nonnegative function is a sum of squares when $V_{\R}(I)$ is
compact. The following example shows that this is not the case.

\begin{example}\label{ex:sphere}
Let $I=(x_1^2+\dots+x_n^2-1)$ be the ideal of the sphere in $\R^n$, with 
$n \geq 4$. According to Theorem 2.6.3 of \cite{Marshall}, there exists a
polynomial function $f$ which is nonnegative on
$V_{\R}(I)$ but not a sum of squares modulo $I$. Since $V(I)$ is nonsingular,
we have by Theorem \ref{thm:nonneg-nbhd} that $f$ is strongly nonnegative
on $V(I)$.
 
If we wish to have an example with $f$ linear, we may simply
add an additional variable $y$, and add to $I$ the relation $y=f$, so that
the resulting coordinate rings are isomorphic. Then $y$ is 
strongly nonnegative, but is not a sum of squares modulo $I$.
\end{example}

\section{Obstructions to theta exactness}

Recall that the closure of the convex hull of a real variety $V_{\R}(I)$ can 
be described as the intersection of all halfspaces defined by linear 
functions nonnegative on it.  Determining a description of the closure of the 
convex hull of a real variety in terms of finitely many polynomial equations 
and inequalities is difficult in general.  To combat this, Gouveia, Parrilo 
and Thomas \cite{GouveiaParriloThomas2008} introduce a hierarchy of nested
spectrahedral shadows containing the convex hull of $V_{\R}(I)$.  The
\textbf{$k$-th theta body} denoted $TH_k(I)$ is precisely
\[
TH_k(I) = \{x \in \R^n : f(x) \geq 0 \ \ \forall \ f \mbox{ linear and }
k\mbox{-sos mod } I\}
\]

These theta bodies form a hierarchy of relaxations
\[
TH_1(I) \supseteq TH_2(I) \supseteq \cdots \supseteq \overline{\conv(V_{\R}(I))}
\]
of the closure of the convex hull of $V_{\R}(I)$.  When the $k$-th theta body coincides with $\overline{\conv(V_{\R}(I))}$, $I$ is said to be \textbf{$TH_k$-exact}. These two concepts are related by the following proposition from 
\cite{GouveiaParriloThomas2008}; see Proposition \ref{prop:strictsos} below 
for a stronger statement.

\begin{proposition}
Let $I \subseteq \rx$ be an ideal.  If $I$ is $(1,k)$-sos then $I$ is $TH_k$-exact.
\end{proposition}

Moreover, Gouveia, Parrilo and Thomas also proved the following remarkable
converse. See Corollary 2.12 of \cite{GouveiaParriloThomas2008}.

\begin{theorem}\label{thm:GPT}
Let $I \subseteq \rx$ be a real radical ideal.  Then $I$ is $(1,k)$-sos if and only if $I$ is $TH_k$-exact.
\end{theorem}

This converse theorem, together with our results on obstructions to an
ideal being $(1,k)$-sos, immediately allow us to rephrase the latter results
in the real radical case in terms of obstructions to theta exactness. We
thus conclude:

\begin{corollary}\label{cor:thetaobstruct}
Let $I \subseteq \rx$ be a real radical ideal.  If there exists a linear function $f$ that is nonnegative on $V_{\R}(I)$ but not strongly nonnegative, then $I$ is not $TH_k$-exact for any $k$. \qed
\end{corollary}

The obstruction theorem of Gouveia and Netzer as they stated it is 
equivalent to the following:

\begin{theorem}\label{thm:not-theta-exact}
Suppose we have $I \subseteq \rx$, and
$P \in V_{\R}(I)$ is a convex-singular point of $V(\sqrt[\R]{I})$, where 
$\sqrt[\R]{I}$
is the real radical ideal associated to $I$. Then $I$ is not $TH_k$-exact
for any $k$.
\end{theorem}

\proof We conclude from Theorem \ref{thm:not-sos} that $\sqrt[\R]{I}$ is
not $(1,k)$-sos, and thus Theorem \ref{thm:GPT} implies that $\sqrt[\R]{I}$
is not $TH_k$-exact. Since $TH_k(\sqrt[\R]{I}) \subseteq TH_k(I)$, we conclude
the desired statement.
\qed

Similarly, we conclude:

\begin{corollary}\label{cor:hypersurface-theta} Suppose $I=(g)$ is principal
and real radical in $\rx$, and suppose $P \in V_{\R}(I)$ is a singularity lying 
on the boundary of $\conv(V_{\R}(I))$. Then $I$ is not $TH_k$-exact for any $k$.
\end{corollary}

As before, Example \ref{ex:squiggly} gives an example in which Corollary
\ref{cor:thetaobstruct} goes further than Theorem \ref{thm:not-theta-exact};
indeed, in this case the ideal is real radical, so we conclude that it is
not $TH_k$-exact for any $k$.

\section{A new sum of squares condition}

Finally, we consider a weaker notion of $(1,k)$-sos arising from strong 
nonnegativity. 

\begin{definition}
Given $d,k \geq 1$, and an ideal $I\subseteq \rx$,
we say that $I$ is \textbf{weakly $(d,k)$-sos} if for every $f \in \rx$ 
of degree at most $d$ which is strongly nonnegative on $V_{\R}(I)$,
we have that $f$ is $k$-sos.
\end{definition}

Though being weakly $(1,k)$-sos relaxes the notion of being $(1,k)$-sos, it still implies $TH_k$-exactness.  This generalizes Lemma 1.5 of \cite{GouveiaParriloThomas2008}.

\begin{proposition}\label{prop:strictsos}
If $I$ is weakly $(1,k)$-sos, then $I$ is $TH_k$-exact.
\end{proposition}

\proof
Let $P \in \R^n$ such that $P \notin \overline{\conv(V_{\R}(I))}$.  By the Separation Theorem, there is a linear polynomial $f$ such that $f$ is nonnegative on $\conv(V_{\R}(I))$ and $f(P) < 0$.  Consider the linear function $g = f - \frac{f(P)}{2}$.  We have $g(P) < 0$, and $g$ is \emph{positive} on $\overline{\conv(V_{\R}(I))}$ and hence positive on $V_{\R}(I)$.  This implies $g$ is strongly nonnegative by Proposition \ref{prop:basic} (2), and $g$ is then a sum of squares of polynomials of degree at most $k$ by hypothesis. Since $P$ was arbitrary outside
$\overline{\conv(V_{\R}(I))}$, the result follows.
\qed

The above proposition along with Theorem \ref{thm:GPT} shows that for real radical ideals, being weakly $(1,k)$-sos is in fact equivalent to being $TH_k$-exact.

\begin{corollary}\label{cor:thetaequiv}
If $I \subseteq \rx$ is a real radical ideal, then the following are equivalent:
\begin{enumerate}
\item $I$ is weakly $(1,k)$-sos
\item $I$ is $(1,k)$-sos
\item $I$ is $TH_k$-exact.
\end{enumerate}
\end{corollary}

\begin{proof}
Proposition~\ref{prop:strictsos} establishes that $(1) \Longrightarrow (3)$.  $(3) \Longrightarrow (2)$ follows from Theorem~\ref{thm:GPT}.  Finally, $(2) \Longrightarrow (1)$ is immediate from Proposition~\ref{prop:basic} (1).
\end{proof}

We conclude by briefly discussing some further questions.
Our original hope was that replacing $(1,k)$-sos with weakly $(1,k)$-sos
would allow the relaxation of the radical portion of the real radical
condition in Theorem \ref{thm:GPT}. We have not yet obtained any results
in this direction, but neither do we have any counterexamples. Indeed,
we are not aware of any examples of an ideal $I$ which is $TH_k$-exact
but not weakly $(1,k)$-sos. It seems implausible that the two conditions
should be equivalent without any sort of hypothesis implying at least
that $V_{\R}(I)$ is Zariski dense in $V_{\C}(I)$, but neither is it entirely
absurd: we note that if $V_{\C}(I)$ is irreducible and $V_{\R}(I)$ is
not Zariski dense, then 
we will have that $V_{\R}(I)$ is contained in the singular locus of $V(I)$.
In addition, if $V_{\R}(I)=\emptyset$, then according to the 
Positivstellensatz (2.2.1 of \cite{Marshall}) we have $-1$ a sum of squares
modulo $I$, which then implies that every polynomial is a sum of squares
modulo $I$.

It would also be interesting to consider effectiveness questions. Of
course, the concept of strong nonnegativity is already useful from an
effectiveness point of view insofar as it provides a new approach to
producing a certificate that a given function is not a sum of squares
modulo an ideal. However, it is also natural to wonder whether, for
instance, it is possible to effectively determine whether a given 
function is strongly nonnegative at a point. This question naturally
breaks up into two subquestions: whether strong nonnegativity at a
point can be effectively computed for a given order of infinitesimal
arcs, and whether for any given ideal $I$, point $P \in V_{\R}(I)$,
and function $f$, one can effectively compute a number $N$ such that to 
determine that $f$ is strongly nonnegative at $P$, it is enough to look
at infinitesimal arcs of order up to $N$. The latter question is 
interesting in and of itself, both from a theoretical point of view,
and because one could envision that even if the answer to the first 
question is negative in general, it could be positive in some more
specific scenarios.

\bibliographystyle{hamsplain}
\bibliography{references}

\end{document}